\documentclass[12pt]{amsart}
\usepackage{graphicx}
\usepackage{amsmath,amssymb}
\usepackage[active]{srcltx} 
\newtheorem{thm}{Theorem}[section]

\newtheorem{lem}[thm]{Lemma}
\newtheorem{prop}[thm]{Proposition}
\newtheorem{exa}[thm]{Example}
\theoremstyle{definition}

\theoremstyle{remark}
\newtheorem{rem}[thm]{Remark}
\numberwithin{equation}{section}

\newcommand{\Zeal}{\mathbb Z}

\begin{document}

\title[]{Coloring link diagrams by Alexander quandles}%
\author{Yongju Bae}%
\address{Department of Mathematics, College of Natural Sciences,
Kyungpook National University, Daegu 702-701, Korea}%
\email{ybae@knu.ac.kr}%

\thanks{The author would like to thank J. Scott Carter and Daniel S. Silver for the discussion and important comments in preparing the article, during his sabbatical visit at University of South Alabama}%

\date{\today}%
\begin{abstract}
In this paper, we study the colorability of link diagrams by the Alexander quandles. We show that
if the reduced Alexander polynomial $\Delta_{L}(t)$ is vanishing, then
$L$ admits a non-trivial coloring by any non-trivial Alexander quandle $Q$, and that
if $\Delta_{L}(t)=1$, then $L$ admits only the trivial coloring by any Alexander quandle $Q$, also show that
if $\Delta_{L}(t)\not=0, 1$, then
$L$ admits a non-trivial coloring by the Alexander quandle $\Lambda/(\Delta_{L}(t))$.
\end{abstract}
\maketitle

\section{Introduction and preliminaries}

In 1982, D. Joyce\cite{Joyce} and S. Matveev\cite{Matveev} defined the notion of quandle.
A {\it quandle} is a non-empty set $X$ equipped with a binary operation
$\ast$ satisfying the following three axioms:
\begin{itemize}
  \item[(Q1)] For any $x \in X$, $x\ast x = x$.
  \item[(Q2)] For any $x, y\in  X$, there is a unique element $z\in X$ such that $x=z\ast y$.
  \item[(Q3)] For any $x, y, z\in  X$, $(x\ast y)\ast z = (x\ast z)\ast (y\ast z)$.
\end{itemize}

The property (Q2) is equivalent to the following property that
\begin{itemize}
  \item[(Q2$'$)] There is a binary operation $\bar\ast:X\times X\to X$ such that for any $x, y\in  X$,
$(x\ast y)\bar\ast y = (x\bar\ast y)\ast y=x$.
\end{itemize}
A quandle homomorphism is a map between two quandles preserving the quandle operation.

There are many quandles, see~\cite{Brieskorn}\cite{CarterJKS}\cite{CarterKS}.
For example, let $X$ be a subset of a group closed under conjugations. Then $X$ is
a quandle, called a {\it conjugation quandle}, under the operation $x\ast y = y^{-1}xy$ for all $x, y\in  X$.

An important class of quandles are Alexander quandles.
Let $\Lambda=\Zeal[t,t^{-1}]$ be the Laurent polynomial ring over the integers. Then any $\Lambda$-module $M$ has a quandle structure, called an {\it Alexander quandle}, under the operation $a\ast b=ta+(1-t)b$ for $a,b\in M$.
An Alexander quandle is said to be {\it finitely generated} if its underlying abelian group is finitely generated.

A {\it coloring} on an oriented classical knot diagram $D$ by a quandle $Q$ is a function $C :R\to Q$, where $R$ is the set of over-arcs in the diagram, satisfying the condition
depicted in the Fig.~\ref{coloring}. In the figure, a crossing with over-arc, $r$, has color $C(r) = y \in Q$. The
under-arcs are called $r_1$ and $r_2$ from top to bottom; they are colored $C(r_1) = x$ and $C(r_2) = x\ast y$.
Note that
locally the colors do not depend on the orientation of the under-arc, and that any constant function $C_q :R\to Q$ at $q\in Q$ is a coloring by the Axiom (Q1), called the {\it trivial coloring} at $q\in Q$.
\begin{figure}[ht]\label{coloring}
  \includegraphics[width=2.5cm]{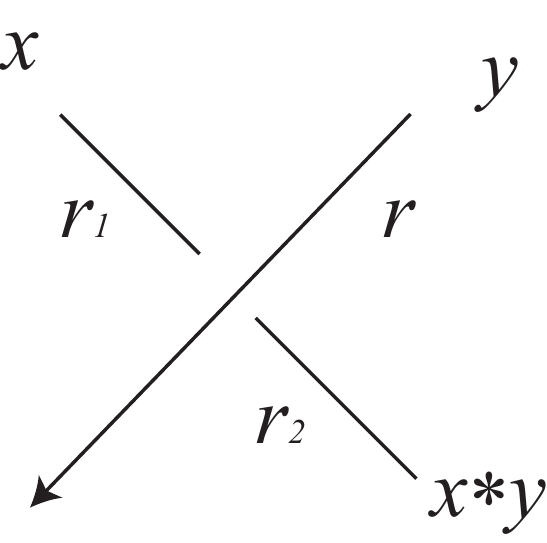}
\end{figure}

In 2001, A. Inoue\cite{Inoue}  proved the following proposition.
\begin{prop}\label{Inoue} Let $p$ be a prime number, $J$ an ideal of the ring $\Lambda_p$ and $Q(K)$ a knot quandle. For each $i\geq 0$, we put $e_i(t)=\Delta_K^{(i)}(t)/\Delta_K^{(i+1)}(t)$, where $\Delta_K^{(i)}(t)$ is the $i$-th Alexander polynomial of $K$ defined by the greatest common divisor polynomial of all $(n-i+1)$-th minor determinants of the Alexander matrix $A_D$ obtained from the Wirtinger presentation. Then the number of all quandle homomorphisms of the knot quandle $Q(K)$ to the Alexander quandle $\Lambda_p/J$ is equal to the cardinality of the module $\Lambda_p/J\oplus\oplus_{i=0}^{n-2}\{\Zeal_p[t,t^{-1}]/(e_i(t),J)\}.$
\end{prop}

As a corollary, he showed that for an Alexander quandle $\Lambda_p/J$ and for a {\it knot} diagram $K$, there are only trivial colorings of a diagram of $L$  by the quandle $\Lambda_p/J$ if and only if the ideal generated by $J$ and $\Delta_K(t)$ is equal to $\Lambda_p$, where $p$ is a prime number.
It is known that, for a finite quandle $X$,  there is a one-to-one
correspondence between quandle homomorphisms $Q(K)\to X$ and colorings $C : R\to X$, see \cite{CarterJKLS}.

In this paper, we will study the colorability of link diagrams by Alexander quandles from the view point of  the reduced Alexander polynomial of the given link. The reduced Alexander polynomial $\Delta_{L}(t)$ of a link $L$ is obtained from the multivariable Alexander polynomial $\Delta_{L}(t_1,\cdots,t_\mu)$ by setting $t_1=\cdots=t_\mu=t,$ where $\mu$ is the number of components of $L$.
The following is the main results.

\begin{thm} Let $\Delta_{L}(t)$ be the reduced Alexander polynomial of a link $L$.
\begin{itemize}
  \item[(1)] If $\Delta_{L}(t)=0$, then
$L$ admits a non-trivial coloring by any non-trivial Alexander quandle $Q$.
  \item[(2)] If $\Delta_{L}(t)=1$, then
$L$ admits only the trivial coloring by any Alexander quandle $Q$.
  \item[(3)] If $\Delta_{L}(t)\not=0, 1$, then
$L$ admits a non-trivial coloring by the Alexander quandle $\Lambda/(\Delta_{L}(t))$.
\end{itemize}
\end{thm}

\section{Linear Algebra with coefficient ring $\Lambda$ and $\Lambda/(f(t))$}

In this section we will study the properties of matrices whose entries are in $\Lambda$ or $\Lambda/(f(t))$, where $f(t)$ is a fixed non-zero polynomial in $\Lambda$.

Consider a system of linear equations whose coefficients are in $\Lambda$ or in $\Lambda/(f(t))$.

{\tiny
\begin{equation}\tag{*}\label{*}
  \left(
    \begin{array}{ccc}
      a_{11}(t) & \cdots & a_{1n}(t) \\
      \vdots & \ddots & \vdots \\
      a_{m1}(t) & \cdots & a_{mn}(t) \\
    \end{array}
  \right)\left(
           \begin{array}{c}
             x_1 \\
             \vdots \\
             x_n \\
           \end{array}
         \right)
=\left(
           \begin{array}{c}
             0 \\
             \vdots \\
             0 \\
           \end{array}
         \right)
\end{equation}}

For a non-zero polynomial $\alpha(t)$ in $\Lambda$, consider the following system $(**)$ of linear equations which is obtained by  multiplying $\alpha(t)$ to the $j$-th row.
{\tiny
\begin{equation}\tag{**}\label{**}
  \left(
    \begin{array}{ccc}
      a_{11}(t) & \cdots & a_{1n}(t) \\
      \vdots & \ddots & \vdots \\
      \alpha(t)a_{j1}(t) & \cdots & \alpha(t)a_{jn}(t) \\
      \vdots & \ddots & \vdots \\
      a_{m1}(t) & \cdots & a_{mn}(t) \\
    \end{array}
  \right)\left(
           \begin{array}{c}
             x_1 \\
             \vdots \\
             x_n \\
           \end{array}
         \right)
=\left(
           \begin{array}{c}
             0 \\
             \vdots \\
             0 \\
           \end{array}
         \right)
\end{equation}}

If we solve the systems with the coefficient ring $\Lambda$, clearly the equations $(*)$ and $(**)$ have the same solution. But if we are working on $\Lambda/(f(t))$, we need to be careful. If $\alpha(t)$ is an unit element $t^n, n\in\Zeal$, then $(\ast)$ and $(\ast\ast)$ have the same solutions. But, if $\alpha(t)$ is not a unit element, the solutions to $(**)$ may not be solutions to $(*)$.
The following lemma gives a characterization for which the equations $(\ast)$ and $(\ast\ast)$ have the same solution.

\begin{prop}\label{solution} $(\ast)$ and $(\ast\ast)$ have the same solution as equations in $\Lambda/(f(t))$ if $\alpha(t)$ and $f(t)$ are relatively prime.
\end{prop}

\begin{proof}
Assume that $\alpha(t)$ and $f(t)$ are relatively prime.
Clearly any solution of $(\ast)$ is a solution of $(\ast\ast).$ Conversely, if $(x_1,\cdots,x_n)^T$ is a solution of $(\ast\ast)$, then,   since $\alpha(t)a_{j1}(t)x_1+\cdots+\alpha(t)a_{jn}(t)x_n=0$ in $\Lambda/(f(t)),$ there is $k(t)\in \Lambda$ such that, in $\Lambda$,
$$\alpha(t)(a_{j1}(t)x_1+\cdots+a_{jn}(t)x_n)=f(t)k(t).$$
Since $\alpha(t)$ and $f(t)$ are relatively prime, $k(t)$ is a multiple of $\alpha(t)$, i.e., $k(t)=k'(t)\alpha(t)$ for some $k'(t)\in \Lambda.$
Hence, in $\Lambda$, $$\alpha(t)(a_{j1}(t)x_1+\cdots+a_{jn}(t)x_n)=f(t)k'(t)\alpha(t),$$ which is equivalent with $$a_{j1}(t)x_1+\cdots+a_{jn}(t)x_n=f(t)k'(t)$$ in $\Lambda$. Thus $(x_1,\cdots,x_n)^T$ is a solution of $(\ast).$
\end{proof}

We remind the reader about matrix theory with coefficients in a field.
It is well-known that the solution of the system $(*)$ of linear equations does not changed by the following elementary row operations, and that every matrix can be changed to a row-echelon form by using elementary row operations.
\begin{itemize}
  \item[$\mathcal{R}_1$] Row switching(a row within the matrix is switched with another row),
  \item[$\mathcal{R}_2$] Row multiplication(each element in a row is multiplied by a non-zero constant) and
  \item[$\mathcal{R}_3$] Row addition(a row is replaced by the sum of that row and a multiple of another row).
\end{itemize}

For the matrices with coefficients in $\Lambda$ or $\Lambda/(f(t))$, we modify the operation $\mathcal{R}_2$ as
\begin{itemize}
  \item[$\mathcal{R}'_2$] Row multiplication for$\Lambda$(each element in a row is multiplied by a non-zero polynomial $\alpha(t)$)
  \item[$\mathcal{R}'_2$] Row multiplication for$\Lambda/(f(t))$(each element in a row is multiplied by a non-zero polynomial $\alpha(t)$ with $(\alpha(t),f(t))=1$)
\end{itemize}

By Lemma~\ref{solution}, one can see that elementary row operations $\mathcal{R}_1, \mathcal{R}'_2$ and $\mathcal{R}_3$  do not change the solution of the system $(*)$, and that every matrix with coefficients in $\Lambda$ or $\Lambda/(f(t))$ can be changed by a row-echelon form by using elementary row operations $\mathcal{R}_1, \mathcal{R}'_2$ and $\mathcal{R}_3$.
We will use the terminology ``rank'' for the matrices with coefficients in $\Lambda$ or $\Lambda/(f(t))$, too.

Consider the system $(*')$ of linear equations whose coefficient matrix $A$ is in a row-echelon form.
{\tiny
\begin{equation}\tag{$*'$}\label{*'} A\text{\bf x}=
  \left(
    \begin{array}{ccccccc}
      a_1(t)    & \ast      & \ast  &\cdots & \ast  &\cdots  &\ast\\
      0         &  0        & a_2(t)&\cdots & \ast  &\cdots  &\ast\\
      \vdots    & \vdots    & \vdots&\ddots &\vdots &\vdots  &\vdots\\
      0         & 0         &\cdots & 0     & a_k(t)&\cdots  &\ast\\
      0         & 0         &\cdots & 0     & 0     &\cdots  &0\\
      0         & 0         &\cdots & 0     & 0     &\cdots  &0\\
    \end{array}
  \right)\left(
           \begin{array}{c}
             x_1 \\
             \vdots \\
             x_n \\
           \end{array}
         \right)
=\left(
           \begin{array}{c}
             0 \\
             \vdots \\
             0 \\
           \end{array}
         \right)
\end{equation}}

\begin{prop}\label{rank-solution}
\begin{itemize}
  \item[(1)] If $\text{rank}(A)<n$, $(*')$ has infinitely many solutions.
  \item[(2)] If $\text{rank}(A)=n$ and if the coefficient ring is $\Lambda$, $(*')$ has the only trivial solution.
  \item[(3)] If $\text{rank}(A)=n$ and if the coefficient ring is $\Lambda/(f(t))$, then
\begin{itemize}
  \item[(i)] if $a_i(t)$ and $f(t)$ are relatively prime for all $i$, then $(*')$ has the only trivial solution;
  \item[(ii)] if there exist $i$ such that $a_i(t)$ and $f(t)$ are not relatively prime, then
$(*')$ has infinitely solutions.
\end{itemize}
\end{itemize}
\end{prop}
\begin{proof} The result is trivial except the case (3)(ii). For the case (3)(ii), consider a linear equation $(t-1)x=0$
in the coefficient ring $\Lambda/((t-1)(t^2+1))$. Even though , the equation $\text{rank}(A)$ is 1 which is the number of indeterminants, it has infinitely many solutions: $x=(t^2+1)g(t)$ for any  $g(t)\in\Lambda$. By modifying the idea in this example, one can find infinitely solutions for $(*')$ in the case (3)(ii).
\end{proof}

\begin{prop}\label{reducing} For $A=(a_{ij})$ an $n\times n$-matrix with coefficients in $\Lambda$, if $a_{11}\not=0$, then define $$b_{ij}(t)=\det\left(\begin{array}{cc}
                      a_{11}(t) & a_{1j}(t) \\
                      a_{j1}(t) & a_{ij}(t) \\
                      \end{array}
                       \right)$$ for $i,j$ with $2\leq i\leq m$ and $2\leq j\leq n$. Let $B$ be the $(n-1)\times (n-1)$-matrix defined by $B=(b_{ij})$.
Then $$\det(B)=a_{11}^{n-2}\det(A).$$
\end{prop}

\begin{proof} By applying the elementary row operation $\mathcal{R}'_2$ (multiply $a_{11}$ to the second row) and then
applying the elementary row operation $\mathcal{R}_3$ (between the first row and the second row), one can obtained the second row of the right matrix below. Notice that to do this, we need to use auxiliary row $a_{21}\times$(the first row).
By applying repeatedly these operations to the remaining rows, we get the matrix on the right below.
{\tiny
\begin{equation*}
  \left(
    \begin{array}{cccc}
      a_{11}(t) & a_{12}(t)&\cdots & a_{1n}(t) \\
      a_{21}(t) & a_{22}(t)&\cdots & a_{2n}(t) \\
          \vdots& \vdots& \ddots & \vdots \\
      a_{n1}(t) & a_{n2}(t)&\cdots & a_{nn}(t) \\
    \end{array}
  \right)\Rightarrow
  \left(
    \begin{array}{cccc}
      a_{11}(t) & a_{12}(t)&\cdots & a_{1n}(t) \\
      0 & b_{22}(t)&\cdots & b_{2n}(t) \\
          \vdots& \vdots& \ddots & \vdots \\
      0 & b_{n2}(t)&\cdots & b_{nn}(t) \\
    \end{array}
  \right)
\end{equation*}}
Since $\mathcal{R}_3$ operation does not change the determinant and such a $\mathcal{R}'_2$ operation gives a multiple of $a_{11}$ to the determinant, we have the result by comparing the determinants of both sides.
\end{proof}

\begin{rem}\label{reducing1} In the above proposition, if the entries in the first column have common divisor $d(t)$, i.e., $a_{11}=a_{11}'(t)d(t), \cdots, a_{n1}=a_{n1}'(t)d(t)$, we can define
$$b_{ij}(t)=\det\left(\begin{array}{cc}
                      a_{11}'(t) & a_{1j}(t) \\
                      a_{j1}'(t) & a_{ij}(t) \\
                      \end{array}
                       \right)$$
In this case, $$a_{11}\det(B)=(a_{11}')^{n-1}\det(A).$$
\end{rem}

\section{Colorability by Alexander quandles}

Let $Q$ be an Alexander quandle.
Let $\phi:B_n\to GL(n,\Lambda)$ be the Burau representation for the braid group $B_n$, which is defined by
$$\phi(\sigma_i)=\left(
                  \begin{array}{cccc}
                    I_{n-1}      & \text{\bf 0} & \text{\bf 0} & \text{\bf 0} \\
                    \text{\bf 0} & 0            & 1            & \text{\bf 0} \\
                    \text{\bf 0} & t            & 1-t          & \text{\bf 0} \\
                    \text{\bf 0} & \text{\bf 0} & \text{\bf 0} & I_{n-i-2} \\
                  \end{array}
                \right), i=1,\cdots, n-1.$$

For a braid $w\in B_n$ , suppose that the diagram of $w$, presented by the standard generators, is colored by $Q$.
By reading the colors assigned to the top arcs of $w$ in the given order, we get an element $(c_1,c_2,\cdots,c_n)$ in $Q^n$.
Similarly, by reading the colors assigned to the bottom arcs of $w$ in the given order, we get another element $(d_1,d_2,\cdots,d_n)$ in $Q^n$. We call the element $(c_1,c_2,\cdots,c_n)$ associated with the top arcs of $w$ the {\it coloring of the braid} $w\in B_n$.
By comparing the definition of the coloring and the definition of the Burau representation, one can see that
$$\phi(w)(c_1,c_2,\cdots,c_n)^T=(d_1,d_2,\cdots,d_n)^T.$$
Notice that the coloring $(c_1,c_2,\cdots,c_n)$ of $w$ induces a coloring of the closure $\overline{w}$ if and only if $(c_1,c_2,\cdots,c_n)=(d_1,d_2,\cdots,d_n)$.

For $A\in M(n,\Lambda)$ an $n\times n$-matrix, we put $$E(A)=\{ (c_1,c_2,\cdots,c_n)\in Q^n | A(c_1,c_2,\cdots,c_n)^T=(c_1,c_2,\cdots,c_n)^T\}.$$
Then for any $w\in B_n$, $E(\phi(w))\not=\emptyset$ because
$\phi(w)(c,c,\cdots,c)=(c,c,\cdots,c)$ for all $c\in Q$.

Since $\phi(w)\in GL(n,\Lambda)\subset M(n,\Lambda)$ and $Q$ a $\Lambda$-module, $\phi(w)$ can be seen as a module homomorphism $\phi(w):Q^n\to Q^n$ defined by the matrix multiplication $\phi(w)(x)=\phi(w)x$ for all $x\in Q^n$. Since $E(\phi(w))$ is the kernel of $(\phi(w)-id)$, it is a submodule of $Q^n$.

\medskip
Let $D$ be a diagram of a link $L$ and $Q$ an Alexander quandle.
Let $\mathcal{VC}_D(Q)$ be the set of all colorings on a link diagram $D$ by $Q$.
Define the addition of colorings and scalar multiplication by
adding the quandle elements assigned on each arc of $D$ by $\mathcal{C}_1$ and $\mathcal{C}_2$
and by multiplying $\alpha(t)$ to the quandle elements assigned on each arc of $D$ by $\mathcal{C}_1$
where $\mathcal{C}_1$ and $\mathcal{C}_2$ are two colorings on a diagram $D$ and $\alpha(t)$ is in $\Lambda$. In~\cite{Inoue}, A. Inoue mentioned about the sum of colorings and scalar multiplication to colorings by the Alexander quandle $\Lambda_p/J$.

It is easy to show that $\mathcal{VC}_D(Q)$ is a $\Lambda$-module under the addition and the scalar multiplication, and is isomorphic to the submodule $E(\phi(w))$ of $\phi(w)$ of $Q^n$, where $D$ is the diagram of the closure $\overline{w}$ given by the braid diagram.

\bigskip

\begin{lem}\label{trivialcoloring} Let $\phi:B_n\to GL(n,\Lambda)$ be the Burau representation and $Q$ any non-trivial Alexander quandle.
If the closure of a braid $w\in B_n$ admits only the trivial coloring by $Q$, then $\text{rank}(\phi(w)-id)=n-1.$
In particular,
if $Q$ is torsion-free and finitely generated as a $\Lambda$-module, the converse holds.
\end{lem}

\begin{proof}
Suppose that  $\text{rank}(\phi(w)-id)<n-1$. By Proposition~\ref{rank-solution}, there exist $x_1(t), \cdots, x_n(t)\in\Lambda$, not all equal, such that
non-zero polynomials of $x_1(t), \cdots, x_n(t)$ are relatively prime and
{\tiny
\begin{equation}\label{XX}
  \left(
    \begin{array}{ccc}
      a_{11} & \cdots & a_{1n} \\
      \vdots & \ddots & \vdots \\
      a_{n1} & \cdots & a_{nn} \\
    \end{array}
  \right)\left(
           \begin{array}{c}
             x_1(t) \\
             \vdots \\
             x_n(t) \\
           \end{array}
         \right)
=\left(
           \begin{array}{c}
             0 \\
             \vdots \\
             0 \\
           \end{array}
         \right),
\end{equation}}
 where $\phi(w)-id=\left(
    \begin{array}{ccc}
      a_{11} & \cdots & a_{1n} \\
      \vdots & \ddots & \vdots \\
      a_{n1} & \cdots & a_{nn} \\
    \end{array}
  \right)$

Let $q\in Q$ be any non-zero element.
Since $<q>=\Lambda q\not=\{0\}$, there is $x(t)\in\Lambda$ such that $x(t)q\not=0.$
Since $\text{rank}(\phi(w)-id)<n-1$, by taking $x(t)$ as a solution corresponding to one of free-variables of $(\ref{XX})$, without loss of generality we may assume that $x_1(t)q,\cdots, x_n(t)q$ are not all zero.
Similarly one can assume that $x_1(t)q,\cdots, x_n(t)q$ are not all equal.
Since
{\tiny
\begin{equation*}
  \left(
    \begin{array}{ccc}
      a_{11}(t) & \cdots & a_{1n}(t) \\
      \vdots & \ddots & \vdots \\
      a_{n1}(t) & \cdots & a_{nn}(t) \\
    \end{array}
  \right)\left(
           \begin{array}{c}
             x_1(t)q \\
             \vdots \\
             x_n(t)q \\
           \end{array}
         \right)
=\left(
           \begin{array}{c}
             a_{11}(t)x_1(t)q+\cdots+a_{1n}(t)x_1(t)q \\
             \vdots \\
             a_{n1}(t)x_1(t)q+\cdots+a_{nn}(t)x_1(t)q \\
           \end{array}
         \right)
=\left(
           \begin{array}{c}
             0 \\
             \vdots \\
             0 \\
           \end{array}
         \right)
\end{equation*}}
 the coloring $(x_1(t)q,\cdots, x_n(t)q)$ of $w$ by $Q$ gives a non-trivial coloring of $\overline{w}$ by $Q$.

Suppose that $Q$ is finitely generated and torsion-free as a $\Lambda$-module
Since $Q$ is $\Lambda$-module which is an abelian group with a scalar multiplication, if $Q$ is finitely generated as an abelian group, by the classification theorem of abelian groups, we can see $Q$ as $\Lambda\oplus\cdots\Lambda\oplus\Lambda/J_{m_1}\oplus\cdots\oplus\Lambda/J_{m_k}$ where $J_{m_i}$ is an ideal.
Since $Q$ is torsion-free, $Q=\Zeal\oplus\cdots\oplus\Zeal$.

Suppose that $\overline{w}$ admits a non-trivial coloring
$((q_{11},\cdots,q_{1m}),\cdots,$ $(q_{n1},$ $\cdots,q_{nm})$ by $Q$, i.e.,
{\tiny $$(\phi(w)-id)\left(
           \begin{array}{c}
             (q_{11},\cdots,q_{1m}) \\
             \vdots \\
             (q_{n1},\cdots,q_{nm}) \\
           \end{array}
         \right)
=\left(
           \begin{array}{c}
             (0,\cdots,0) \\
             \vdots \\
             (0,\cdots,0) \\
           \end{array}
         \right)
$$}
Since $((q_{11},\cdots,q_{1m}),\cdots, (q_{n1},\cdots,q_{nm})$ is non-trivial coloring, there exists $j$ such that
$q_{j1}, \cdots, q_{jn}$ are not all equal. Since, in $\Lambda$,{\tiny $$\phi(w)\left(
           \begin{array}{c}
            q_{j1} \\
             \vdots \\
             q_{jn} \\
           \end{array}
         \right)
=\left(
           \begin{array}{c}
             q_{j1} \\
             \vdots \\
             q_{jn} \\
           \end{array}
         \right)
$$}
and since the equation $(\phi(w)-id){\bf x}={\bf 0}$
has two linearly independent solutions $(q_{j1},\cdots,q_{jn})$ and $(1,\cdots,1)$, $\text{rank}(\phi(w)-id)<n-1$ by Proposition~\ref{rank-solution}.
\end{proof}

It is known that, for a braid word $w\in B_n$,
\begin{eqnarray*}
  &&\phi(w) = C
\left[ \begin {array}{cc} \widetilde{\phi} \left( w \right) &\ast
\\\noalign{\medskip}0&1\end {array} \right]
C^{-1} \ \text{for some}\ \ C\in GL(n,\Lambda), \ \text{and}\\
 &&\det(\widetilde{\phi}(w)-id) = (1+t+\cdots+t^{n-1})\Delta_{L}(t)
\end{eqnarray*}
where $L$ is the closure of $w$ and $\Delta_{L}(t)$ is the reduced Alexander polynomial of the link $L$, see\cite{Birman},\cite{BurdeZieschang}.

\begin{thm}\label{zero-Alex} If the reduced Alexander polynomial $\Delta_{L}(t)$ of a link $L$ is zero, then
$L$ admits a non-trivial coloring by any Alexander quandle $Q$.
\end{thm}

\begin{proof}  Suppose that $L$ is the closure of $w\in B_n$.
Let $\phi:B_n\to GL(n,\Lambda)$ be the Burau representation.
Since $\phi(w)-id=C
\left[ \begin {array}{cc} \widetilde{\phi} \left( w \right)-id &\ast
\\\noalign{\medskip}0&0\end {array} \right]
C^{-1}$ and $\det(\widetilde{\phi}(w)-id)=(1+t+\cdots+t^{n-1})\Delta_{L}(t)$, $\Delta_{L}(t)=0$ if and only if
$\det(\widetilde{\phi}(w)-id)=0$ if and only if $\text{rank}(\phi(w)-id)=\text{rank}(\widetilde{\phi}(w)-id)\leq n-2$.

By Lemma~\ref{trivialcoloring}, $L$ admits a non-trivial coloring by any Alexander quandle $Q$.
\end{proof}

\begin{exa} {\rm $L9n27$ is the closure of $w=\sigma_3^{-1}\sigma_2^{-1}\sigma_1^2\sigma_2^{-1}\sigma_3\sigma_2\sigma_1^{-1}\sigma_2\sigma_1^{-1}\sigma_2\in B_4$, and its Alexander polynomial is vanishing. By the matrix calculation, one can see that  $(\phi(w)-id)$ can be changed to the following by the elementary row operations $\mathcal{R}_1, \mathcal{R}_2'$ and $\mathcal{R}_3$.
{\tiny $$\left[ \begin {array}{cccc} -{\frac {-1+4\,t-3\,{t}^{2}+{t}^{3}}{t}}&
-1+4\,t-3\,{t}^{2}+{t}^{3}&-{\frac {2-6\,t+7\,{t}^{2}-4\,{t}^{3}+{t}^{
4}}{t}}&-{\frac {-1+t}{t}}\\\noalign{\medskip}0&0&{\frac { \left( -1+t
 \right) ^{3}}{{t}^{2}}}&-{\frac { \left( -1+t \right) ^{3}}{{t}^{2}}}
\\\noalign{\medskip}0&0&0&0\\\noalign{\medskip}0&0&0&0\end {array}
 \right]
$$}
Since $\text{rank}(\phi(w)-id)=2$, $(\phi(w)-id)(x_1,x_2,x_3,x_4)^T=(0,0,0,0)^T$ has non-constant solutions, e.g., $(t,1,0,0)$
and $(1-2t,-1,1,1)$.
For any Alexander quandle $Q$ and $q\in Q$,
by coloring the top strands of the braid $w$ by $(1-2t)q,-q,q$ and $q$ in the given order, one can obtain a non-trivial coloring of $L9n27$ by $Q$ whenever $(1-2t)q,-q,q$ are not all equal in $Q$.}
\end{exa}

It is well-known that if $L$ is a split link, then $\Delta_{L}(t)=0$, but the converse does not hold.
There are 11 prime links with up to 11 crossings whose multi-variable Alexander polynomial is 0:  $L9n27$, $L10n32$, $L10n36$, $L10n107$, $L11n244$, $L11n247$, $L11n334$, $L11n381$, $L11n396$, $L11n404$ and $L11n406$ in Thistlethwaite Link Table. One can see that $\dim{E(\phi(w))}= 2$ for the above 11 links, where $w$ is the braid presentation of the link given in Thistlethwaite Link Table.
\bigskip

Now, assume that the reduced Alexander polynomial $\Delta_{L}(t)$ of $L$ is non-vanishing.
If $\Delta_{L}(t)=1$, $L$ admits only the trivial coloring by the quandle $\Lambda/(\Delta_{L}(t))$ because  $\Lambda/(\Delta_{L}(t))=\{1\}$.

\begin{thm} Let $L$ be a link with non-trivial and non-vanishing reduced Alexander polynomial $\Delta_{L}(t)$.
Then $L$ admits a non-trivial coloring by the Alexander quandle $\Lambda/(\Delta_{L}(t))$.
\end{thm}

\begin{proof} Let $w\in B_n$ be a braid presentation of $L$. Let $\phi : B_n\to GL(n,\Lambda)$ be the Burau representation.
Observe that the sum of entries in each row of the matrix $(\phi(w)-id)$ is zero.
Since $\Delta_{L}(t)\not=0$, $\text{rank}(\phi(w)-id)=n-1$, so that $(\phi(w)-id)$ can be changed to the matrix of the following form by elementary row operations $\mathcal{R}_1, \mathcal{R}_2', \mathcal{R}_3$, see Proposition~\ref{reducing} and Remark~\ref{reducing1}.

$$
    \begin{bmatrix}
    a_{11}(t) & a_{12}(t) & a_{13}(t) &\cdots & a_{1(n-1)}(t) & a_{1n}(t) \\
    0         & a_{22}(t) & a_{23}(t) &\cdots & a_{2(n-1)}(t) & a_{2n}(t) \\
    0         &      0    & a_{33}(t) &\cdots & a_{3(n-1)}(t) & a_{3n}(t) \\
    \vdots    & \vdots    & \vdots    &\ddots & \vdots & \vdots    \\
    0         &      0    &     0     &\cdots & a_{(n-1)(n-1)}(t) &  a_{(n-1)n}(t) \\
    0         &      0    &     0     &\cdots &  0 &  0
  \end{bmatrix}
$$
Since the sum of each row entries is zero, it is enough to solve the following system:
{\tiny
\begin{equation}\label{echelon}
\begin{bmatrix}
    a_{11}(t) & a_{12}(t) & a_{13}(t) &\cdots & a_{1(n-1)}(t) \\
    0         & a_{22}(t) & a_{23}(t) &\cdots & a_{2(n-1)}(t) \\
    0         &      0    & a_{33}(t) &\cdots & a_{3(n-1)}(t) \\
    \vdots    & \vdots    & \vdots    &\ddots & \vdots \\
    0         &      0    &     0     &\cdots & a_{(n-1)(n-1)}(t) \\
  \end{bmatrix}\left(
           \begin{array}{c}
            x_1(t) \\
             \vdots \\
             x_{n-1}(t) \\
           \end{array}
         \right)
=\left(
           \begin{array}{c}
             0 \\
             \vdots \\
             0 \\
           \end{array}
         \right),
\end{equation}
}

By applying Proposition~\ref{reducing} inductively, we have
\begin{eqnarray*}
   && a_{11}(t)a_{22}(t)\cdots a_{(n-1)(n-1)}(t) \\
  && \qquad =a_{11}(t)^{n-2}a_{22}(t)^{n-3}\cdots a_{(n-2)(n-2)}(t)^{1}\det(\widetilde{\phi}(w)-id)
\end{eqnarray*}

Since $\det(\widetilde{\phi}(w)-id)=\Delta_L(t)(1+t+\cdots+t^{n-1})$,
there exists $j$ such that $a_{jj}(t)$ is not relatively prime with $\Delta_L(t)$.
If $a_{jj}(t)$ is a multiple of $\Delta_L(t)$, it is zero in $\Lambda/(\Delta_L(t))$ so that
the rank of the coefficient matrix is less than $n-1$.
By Proposition~\ref{solution}(1), the system \ref{echelon} has infinitely many solutions.

If all diagonal entries which are not relatively prime with $\Delta_L(t)$ are not a multiple of $\Delta_L(t)$,
then the rank of the coefficient matrix is $n-1$. By Proposition~\ref{solution}(3)(ii), the system \ref{echelon} has infinitely many solutions.
\end{proof}

\begin{rem} From the proof of the above theorem, one can see that
\begin{itemize}
  \item[(1)] if $f(t)$ is a factor of $\Delta_{L}(t)$, then $L$ admits a non-trivial coloring by the Alexander quandle $\Lambda/(f(t))$.
  \item[(2)] in particular, if $f(t)$ is irreducible, then only one diagonal entry will be zero in $\Lambda/(f(t))$ so that $\text{rank}(\phi(w)-id)=n-2$ in $\Lambda/(f(t))$.
  \item[(3)] if $\text{rank}(\phi(w)-id)=n-2$ in $\Lambda/(f(t))$ and if $(x_1(t),\cdots,x_n(t))$ is a non-trivial coloring of $\overline{w}$ by $\Lambda/(f(t))$, all colorings of $\overline{w}$ by $\Lambda/(f(t))$ are linear combinations of $(x_1(t),\cdots,x_n(t))$ and $(1,\cdots,1)$.
\end{itemize}
\end{rem}

\begin{exa} {\rm
 The knot $8_{15}$ is the closure of $w=\sigma_1^2\sigma_2^{-1}\sigma_1\sigma_3\sigma_2^3\sigma_3$ in $B_4$ and its Alexander polynomial is $\Delta_{8_{15}}(t)=3-8t+11t^2-8t^3+3t^4=(3t^2-5t+3)(1-t+t^2)$. Notice that
$(\phi(w)-id)$ can be changed to the following matrix by elementary row operations $\mathcal{R}_1, \mathcal{R}_2', \mathcal{R}_3$.
{\tiny $$\left[ \begin {array}{cccc} -2+2\,t-{t}^{2}&- \left( -2\,{t}^{2}-1+3
\,t+{t}^{3} \right)  \left( -1+t \right) &1-t&6\,{t}^{2}+2-5\,t-3\,{t}
^{3}+{t}^{4}\\\noalign{\medskip}0& \left( -t+1+{t}^{2} \right)
 \left( {t}^{2}-3\,t+3 \right) &t-1-{t}^{2}&-6\,{t}^{2}+5\,t-2+4\,{t}^
{3}-{t}^{4}\\\noalign{\medskip}0&0&3\,{t}^{2}-5\,t+3&5\,t-3\,{t}^{2}-3
\\\noalign{\medskip}0&0&0&0\end {array} \right]
$$}
Note that every diagonal entries of the above matrix can not be zero in $\Lambda/(\Delta_{8_{15}}(t))$ so that $\text{rank}(\phi(w)-id)=n-1.$ Since the second and the third diagonal entries are not relatively prime with $\Delta_{8_{15}}(t)$,  $8_{15}$ admits infinitely many non-trivial colorings by $\Lambda/(\Delta_{8_{15}}(t))$ by Proposition~\ref{solution}.
For example, the colorings $(5t-3t^2-3, (-1+t)(3t^2-5t+3), 0, 0)$ and $(t(5t-3t^2-3), 5t-3t^2-3, 0, 0) $ of $w$ give non-trivial colorings of $8_{15}.$

For the irreducible factor $(1-t+t^2)$ of $\Delta_{8_{15}}(t)$, only the second diagonal entry is zero in $\Lambda/(1-t+t^2)$ so that $\text{rank}(\phi(w)-id)=n-2.$ Note that $(1, 1-t, 0, 0)$ is a non-trivial coloring of $8_{15}$ by $\Lambda/(1-t+t^2)$, and that all colorings of $8_{15}$ by $\Lambda/(1-t+t^2)$ are linear combinations of $(1, 1-t, 0, 0)$ and $(1,1,1,1)$.

For the other irreducible factor $(3t^2-5t+3)$ of $\Delta_{8_{15}}(t)$, one can obtain the similar result.
Indeed, $(-t^2+1,1,t^2-3 t+3,0)$ is a non-trivial coloring of $8_{15}$ by $\Lambda/(3-5t+3t^2)$, and all colorings of $8_{15}$ by $\Lambda/(3-5t+3t^2)$ are linear combinations of $(-t^2+1,1,t^2-3 t+3,0)$ and $(1,1,1,1)$.}
\end{exa}

\begin{rem}
Suppose that $f(t)$ is an irreducible factor of $\Delta_{L}(t)$  with multiplicity $k$.
Then $f(t)$ can be distributed over the $k$ diagonal entries of the matrix (\ref{echelon}), in maximum, so that
$$n-k-1\leq\text{rank}(\phi(w)-id)\leq n-2.$$
In particular, if $f(t)$ is with multiplicity $1$, then $\text{rank}(\phi(w)-id)= n-2.$
\end{rem}

\begin{exa}{\rm
(1) For the trefoil knot $3_1=\overline{\sigma_i^3}$, $\Delta_{3_1}(t)=1-t+t^2$ is irreducible, and hence $\text{rank}(\phi(w)-id)= 0.$
Hence one can choose any non-constant element, say $(1,0)$, in $(\Lambda/(\Delta_{3_1}(t)))^2$  as a non-trivial coloring of $3_1$ (the first arc is colored by $1$ and the second by $0$). Since
$\{(1,0),(1,1)\}$ generates $(\Lambda/(\Delta_{3_1}(t)))^2$, every element of $(\Lambda/(1-t+t^2))^2$ can be a coloring of $3_1$.

(2) The knot $8_{20}$ is the closure of $w=\sigma_1^3\sigma_2^{-1}\sigma_1^{-3}\sigma_2^{-1}$ in $B_4$ and $$\Delta_{8_{20}}(t)=1-2t+3t^2-2t^3+t^4=(1-t+t^2)^2.$$ One can see that
$(\phi(w)-id)$ can be changed to the following row-echelon form by elementary row operations $\mathcal{R}_1, \mathcal{R}_2', \mathcal{R}_3$.
{\tiny $$ \left[ \begin {array}{ccc} {\frac {-t+1+{t}^{2}}{{t}^{2}}}&{\frac {
 \left( -1+t \right) ^{2}}{{t}^{3}}}&-{\frac {1-t+{t}^{3}}{{t}^{3}}}
\\\noalign{\medskip}0&-{\frac {-t+1+{t}^{2}}{{t}^{2}}}&{\frac {-t+1+{t
}^{2}}{{t}^{2}}}\\\noalign{\medskip}0&0&0\end {array} \right]
$$}
Note that $\text{rank}(\phi(w)-id)= 2$ in $\Lambda/(\Delta_{8_{20}}(t))$. By Proposition~\ref{solution}, $8_{20}$ admits a non-trivial coloring by $Q=\Lambda/(\Delta_L(t))$, e.g., $(1-t+t^2, 0, 0)$.
Note that, in $\Lambda/(1-t+t^2)$, $\text{rank}(\phi(w)-id)= 1$  and that $(1, 0, 0)$ is a non-trivial coloring of $8_{20}$ by $Q=\Lambda/(1-t+t^2)$, and hence every coloring of $8_{20}$ by $Q=\Lambda/(1-t+t^2)$ is of the form
$\alpha(t)(1,0,0)+\beta(t)(1,1,1)=(\alpha(t)+\beta(t),\beta(t),\beta(t)).$}
\end{exa}

\newpage

\noindent{\bf Non-Trivial Coloring Table.}

\noindent The table in the last page is a list of non-trivial colorings, in which we used the braid notations and Alexander polynomials in {\it KnotInfo: Table of Knot Invariants~
\cite{ChaLivingston}}. In the table, the tuple in the column ``non-trivial coloring by the Alexander quandle $\Lambda/(\Delta_L(t))$''  denotes a coloring of the  braid, by the Alexander quandle $\Lambda/(\Delta_L(t))$, whose top arcs are colored by the tuple in the given order. Such a coloring of the braid gives a coloring of the link. All other colorings are obtained by linear combinations of them with the trivial coloring $(1,1,\cdots,1).$
\medskip

\centerline{\bf Non-Trivial Coloring Table}

\centerline{\tiny \begin{tabular}{|c|l|}
  \hline
   Name     & non-trivial coloring by Alexander quandle $\Lambda/(\Delta_L(t))$  \\
   \hline
  $3_1$ &  $(1,0)$ \\
  $4_1$ &  $(-1+t,-1+2t,0)$ \\
  $5_1$ &  $(1,0)$ \\
  $5_2$ &  $(3t-2,4-2t,0)$ \\
  $6_1$ &  $(2t-1,-1+t,3t-2,0)$ \\
  $6_2$ &  $(t^3-2t^2+2t-1, 2t^3-2t^2+2t-1, 0)$ \\
  $6_3$ &  $(-t^2+2t-1, -2t^2+2t-1, 0)$ \\
  $7_1$ &  $(1,0)$ \\
  $7_2$ &  $(11t-12,-5t+3,6t-9,0)$ \\
  $7_3$ &  $(t^3-t^2+3t-2, -2t^3+2t^2-2t+4, 0)$ \\
  $7_4$ &  $(-53t+76, 56t-32, -96t+128, 0)$ \\
  $7_5$ &  $(-t^3+2t^2-2t, 4t-4t^2-4+2t^3, 0)$ \\
  $7_6$ &  $(3t^3-6t^2+5t-1, 4t^3-6t^2+5t-1, -4t^2-1+4t+2t^3, 0)$ \\
  $7_7$ &  $(2t^3-7t^2+5t-1, 2t^3-8t^2+5t-1, -6t^2+2t^3-1+4t, 0)$ \\
  $8_1$ &  $(2t-2, 3t-2, -1+t, 4t-3, 0)$ \\
  $8_2$ &  $(t^5-2t^4+2t^3-2t^2+2t-1, t^5-2t^4+2t^3-2t^2+2t-1, 0)$ \\
  $8_3$ &  $(13t-12, 21t-12, 5t-4, 29t-20, 0)$ \\
  $8_4$ &  $(5t^3-t^2+3t-2, t^3-t^2+3t-2, 9t^3-9t^2+11t-6, 0)$ \\
  $8_5$ &  $(t^3-t^2+t-1, 2t^3-t^2+t-1, 0)$ \\
  $8_6$ &  $(2t^3-2t^2+2t-1, t^3-2t^2+2t-1, 3t^3-4t^2+4t-2, 0)$ \\
  $8_7$ &  $(-t^4+2t^3-2t^2+2t-1, -2t^4+2t^3-2t^2+2t-1, 0)$ \\
  $8_8$ &  $(-2t^2+2t-1, -t^2+2t-1, -3t^2+4t-2, 0)$ \\
  $8_9$ &  $(t^3-2t^2+2t-1, 2t^3-2t^2+2t-1, 0)$ \\
  $8_{10}$ &  $(-t^4+2t^3-3t^2+2t-1, -2t^4+3t^3-3t^2+2t-1, 0)$ \\
  $8_{11}$ &  $(3t^3-5t^2+5t-2, 2t^3-5t^2+5t-2, 4t^3-6t^2+5t-2, 0)$ \\
  $8_{12}$ &  $(t^3-4t^2+3t, t^3-4t^2+2t, -4t^2+4t+t^3-1, 2t^3-8t^2+6t-1, 0)$ \\
  $8_{13}$ &  $(-3t^3+3t^2-5t+2, -3t^3+7t^2-5t+2, -3t^3-t^2+3t-2, 0)$ \\
  $8_{14}$ &  $(3t^3-5t^2+5t-2, 2t^3-5t^2+5t-2, 4t^3-7t^2+6t-2, 0)$ \\
  $8_{15}$ &  $(5t-3t^2-3, (-1+t)(3t^2-5t+3), 0, 0),(t(5t-3t^2-3), 5t-3t^2-3, 0, 0)$\\
  $8_{16}$ &  $(-2t^4+3t^3-4t^2+3t-1, t^5-4t^4+5t^3-5t^2+3t-1, 0)$ \\
  $8_{17}$ &  $(2t^3-3t^2+3t-1, -t^4+4t^3-4t^2+3t-1, 0)$ \\
  $8_{18}$ &  $(2t-1, t^3-2t^2+3t-1, 0)$ \\
  $8_{19}$ &  $(t, -t^2+t+1, 0)$ \\
  $8_{20}$ &  $(1-t+t^2, 0, 0)$  \\
  \hline
\end{tabular}}

\end{document}